\documentclass[12pt,english,spanish]{article}
\usepackage{amsmath}
\usepackage[T1]{fontenc}
\usepackage[letterpaper]{geometry}
\usepackage{amsmath}
\usepackage{amssymb}
\usepackage{comment}
\usepackage{graphicx}

\usepackage{epic}
\usepackage{pstricks}
\usepackage{pst-node}

\usepackage{pstricks-add}

\def\BBox{\kern  -0.2cm\hbox{\vrule width 0.2cm height 0.2cm}}

\def\po{\mathcal{P}}
\renewcommand{\phi}{\varphi}

\def\infinite{\mathcal{1}}
\def\fl{\mathcal{F}}
\def\Mon{\mathrm{Mon}}

\newtheorem{lemma}{Lemma}[section]
\newtheorem{theorem}{Theorem}[section]

\newtheorem{propo}{Proposition}[section]

\newtheorem{remark}[theorem]{Remark}

\title{Colorful Polytopes and Graphs}
\author{
Gabriela Araujo-Pardo\!\!
\thanks{garaujo@matem.unam.mx, Supported by CONACYT 57371 and by PAPIIT-M\'exico under project 104609-3.},\\ 
Isabel Hubard\!\!
\thanks{hubard@matem.unam.mx, Supported by SMM-Fundaci\'on Sof\'ia Kovalevskaia and by PAPIIT-M\'exico under project IN106811}, \\
Deborah Oliveros\!\!
\thanks{dolivero@matem.unam.mx, Supported by PAPIIT-M\'exico under project 104609-3.} \\
{\small  Instituto de Matem\'{a}ticas}\\
{\small  Universidad Nacional Aut\'{o}noma de M\'{e}xico, M\'{e}xico}
\\[1ex]
{\small and}\\[1ex]
Egon Schulte\!\!
\thanks{schulte@neu.edu, Supported by NSF-grant DMS-0856675 and by PAPIIT-M\'exico under project 104609-3.}\\
{\small Department of Mathematics}\\
{\small Northeastern University, Boston, USA} }

\begin{document}
\maketitle

\begin{center}
{\em With best wishes for our friend and colleague, Luis Montejano}
\end{center}
\smallskip

\begin{abstract}
The paper investigates connections between abstract polytopes and properly edge colored graphs. Given any finite $n$-edge-colored $n$-regular graph $\mathcal{G}$, we associate to $\mathcal{G}$ a simple abstract polytope $\mathcal{P}_{\mathcal{G}}$ of rank $n$, the {\em colorful polytope\/} of $\mathcal{G}$, with $1$-skeleton isomorphic to $\mathcal{G}$. We investigate the interplay between the geometric, combinatorial, or algebraic properties of the polytope $\mathcal{P}_{\mathcal{G}}$ and the combinatorial or algebraic structure of the underlying graph $\mathcal{G}$, focussing in particular on aspects of symmetry. Several such families of colorful polytopes are studied including examples derived from a Cayley graph, in particular the graphicahedra, as well as the flag adjacency polytopes and related monodromy polytopes associated with a given abstract polytope. The duals of certain families of colorful polytopes have been important in the topological study of colored triangulations and crystallization of manifolds.
\end{abstract}
\noindent
{\bf Key words.} ~ Edge chromatic number, Graphicahedron, Abstract Polytope, Cayley Graph.

\noindent
{\bf MSC 2000.} ~ Primary: 51M20.  Secondary: 05C25, 52B15.


\section{Introduction}
\label{intro}

The present paper studies interesting connections between abstract polytopes and edge-colored graphs. Every properly edge-colored regular graph naturally gives rise to a simple abstract polytope, called its {\em colorful polytope\/}, built by following precise instructions encoded in the graph. We investigate several such families of polytopes and the interplay between their geometric, combinatorial, or algebraic properties and the combinatorial or algebraic structure of the underlying graph.

The key idea of associating a combinatorial structure or a topological space with a properly edge-colored regular graph is not new and has been successfully exploited, usually in the dual setting, in the topological study of colored triangulations and crystallization of manifolds (see Bracho \& Montejano~\cite{bracho}, Ferri, Gagliardi \& Graselli~\cite{fgg}, Pezzana~\cite{pez}, Lins \& Mandel~\cite{lima}, Vince~\cite{vin83a,vin83b}, and K\"uhnel~\cite{kuhn}). Here we take a polytopes approach and investigate these structures as simple abstract polytopes rather than simplicial complexes or triangulations, focussing on their geometric, combinatorial, and algebraic symmetry rather than their topology. 

The paper is organized as follows. In Section~\ref{secbasnot} we review basic concepts for edge colorings of graphs and for abstract polytopes. Then, in Sections~\ref{polsfromgraphs} and~\ref{autgrcol}, we establish that the combinatorial structure derived from a properly edge-colored regular graph is an abstract polytope with $1$-skeleton isomorphic to the given graph, and determine its automorphism group. In Section~\ref{flagpo} we describe a colorful polytope, the {\em flag adjacency polytope\/}, associated with the flag graph of a given abstract polytope; this flag graph is a combinatorial map in the sense~\cite{vin83a,vin83b}. Then, in Section~\ref{colcay}, colorful polytopes of Cayley graphs are investigated. In particular, we revisit the graphicahedron explored in \cite{graph,symgra}. Finally, Section~\ref{mono} explains how the flag-adjacency polytope of a given polytope is related to another colorful polytope, the {\em monodromy polytope\/}, of the monodromy group of the given polytope (see Hartley~\cite{hartley} and Hubard, Orbanic \& Weiss~\cite{d-auto}).

\section{Basic notions}
\label{secbasnot}

We begin by briefly reviewing some of the terminology for graphs and abstract polytopes.  For further basic definitions and terminology see Chartrand \& Lesniak~\cite{CL96} and McMullen \& Schulte~\cite{McMS02}. 

Let $\mathcal{G}$ be a finite connected simple graph, without loops or multiple edges and with vertex set $V$ and edge set $E$.  An {\em edge coloring} of $\mathcal{G}$ is an assignment of colors to the edges of $\mathcal{G}$ such that adjacent edges are colored differently. More precisely, if $R$ is a set of $n$ {\em colors\/}, then an {\em $n$-edge coloring} of $\mathcal{G}$ with {\em color set\/} $R$ is a mapping $c: E\rightarrow R$ such that $c(e)\neq c(f)$ if $e$ and $f$ are adjacent edges of $\mathcal{G}$. Then the pair $(\mathcal{G},c)$, usually simply denoted by $\mathcal{G}$, is called an {\em $n$-edge colored graph\/}. The graph $\mathcal{G}$ is {\em $n$-edge colorable\/} if there exists an $n$-edge coloring of $\mathcal{G}$. The minimum number $n$ for which $\mathcal{G}$ is $n$-edge colorable is called its {\em edge chromatic number}, or {\em chromatic index}, in the literature usually denoted by $\chi _{1}(\mathcal{G})$. 

Let $\Delta(\mathcal{G})$ denote the maximum degree among the vertices of $\mathcal{G}$. Clearly, $\Delta (\mathcal{G})\leq\chi _{1}(\mathcal{G})$. By a well-known result of Vizing~\cite{Vizing}, we also have $\chi _{1}(\mathcal{G})\leq \Delta (\mathcal{G})+1$. Thus there are two kinds of graphs, those with $\chi _{1}(\mathcal{G})=\Delta (\mathcal{G})$, here said to be of {\em type~1\/}, and those with $\chi _{1}(\mathcal{G})=\Delta (\mathcal{G})+1$, here said to be of {\em type~2}. It was proved in Erd\"os \& Wilson~\cite{ErdosWilson} that the probability for a graph on $p$ vertices to be of type~1 approaches 1 as $p$ approaches $\infty$. Hence in general there are considerably more graphs of type 1 than graphs of type 2. However, the problem of determining which graphs are of which type is unsolved.

The edge chromatic number of an $r$-regular graph $\mathcal{G}$ is $r$ or $r+1$. If $\mathcal{G}$ is of type 1, then any $r$-edge coloring determines $r$ edge-disjoint $1$-factors (perfect matchings) of $\mathcal{G}$, each given by one color class of edges; in fact, each vertex of $\mathcal{G}$ lies in exactly one edge of each color. Conversely, it is easy to see that a connected $r$-regular graph whose edge set can be partitioned into $r$\ $1$-factors (that is, a {\em 1-factorable graph\/}) is of type 1. Thus a connected $r$-regular graph is of type 1 if and only if it is 1-factorable. 

Our main interest is in connected $r$-regular graphs $\mathcal{G}=(\mathcal{G},c)$ of type $1$, with vertex set $V$, edge set $E$, and an $r$-edge coloring map $c: E\rightarrow R$ with color set $R$. We refer to these as {\em well} ({\em edge\/}) {\em colored $r$-regular graphs\/}, or simply {\em properly (edge) colored graphs}. We sometimes abuse notation and also denote by $\mathcal{G}$ the underlying ``uncolored" graph of a properly (edge) colored graph. The complete graphs $K_p$ with an even number $p$ of vertices, the $p$-cycles $C_p$ with $p$ even and $p\neq 2$, and the regular bipartite graphs such as the Heawood graph, and some other  nice examples like the Coxeter graph, are all examples of properly (edge) colored graphs (see Figure~\ref{Heawood}). Moreover, every planar connected $r$-regular graph with $r\geq 10$ is a properly (edge) colored graph (see \cite{Vizing2}). 

\begin{figure}[htbp]
\begin{center}
\vspace{-15cm}
\hspace{-3cm}
\includegraphics[width=18cm]{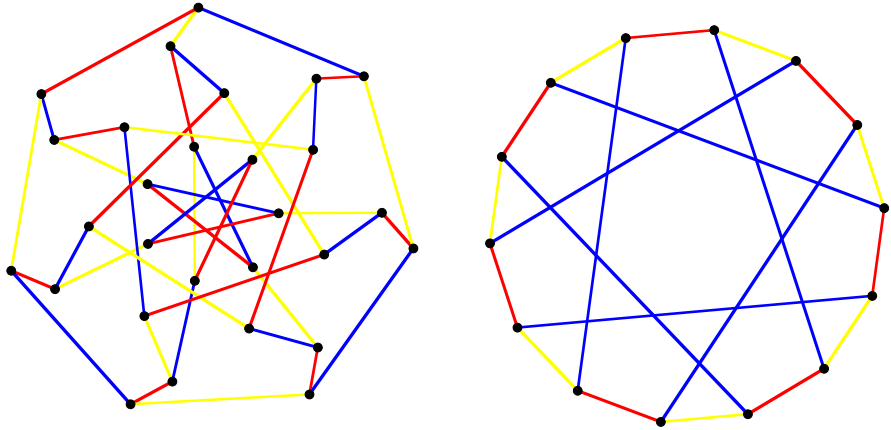} 
\vspace{-2cm}
\caption{Edge colorings of the Coxeter graph and the Heawood graph, each determining a 3-polytope of type $\{14,3\}$.}
\label{Heawood}
\end{center}
\end{figure}

For a properly (edge) colored $r$-regular graph $\mathcal{G}$ we let $\Gamma(\mathcal{G})$ denote the graph automorphism group of the underlying ``uncolored'' graph, that is, the subgroup of permutations of the vertex set $V$ that preserve the edge set $E$. There are two subgroups of $\Gamma(\mathcal{G})$ that are associated with the $r$-edge coloring map  $c: E\rightarrow R$ and are of particular interest for us. They are obtained as follows.

We say that an automorphism $\gamma$ in $\Gamma(\mathcal{G})$ is {\em color preserving\/} if $\gamma$ maps every edge of $\mathcal{G}$ to an edge with the same color; that is, $c(\gamma (e))=c(e)$ for each $e$ in $E$. On the other hand, we say that an automorphism $\gamma$ in $\Gamma(\mathcal{G})$ is {\em color respecting\/} if 
any two edges of $\mathcal{G}$ with the same color are mapped to edges that also have the same color (which may be distinct from the first color); that is, $c(\gamma(e))=c(\gamma(e^{\prime}))$ whenever $e,e^{\prime}$ are edges in $E$ with $c(e)=c(e^{\prime})$. Naturally we now obtain two special subgroups of $\Gamma(\mathcal{G})$, namely 
the subgroup $\Gamma_p(\mathcal{G})$ consisting of all color preserving automorphisms of $\mathcal{G}$, and the subgroup $\Gamma_{c}(\mathcal{G})$ consisting of all color respecting automorphisms of $\mathcal{G}$. Clearly every color preserving automorphism is also color respecting, so $\Gamma_p(\mathcal{G})$ is a (generally proper) subgroup of $\Gamma_c(\mathcal{G})$.
\medskip
 
In this paper we construct abstract polytopes from properly (edge) colored graphs. An ({\em abstract\/}) {\em polytope of rank\/} $n$, or simply an {\em $n$-polytope}, is a partially ordered set $\mathcal{P}$ with a strictly monotone rank function with range $\{-1,0,\ldots,n\}$ satisfying the following conditions. The elements of rank $j$ are called the {\em $j$-faces\/} of $\mathcal{P}$, or {\em vertices}, {\em edges\/} and {\em facets\/} of $\mathcal{P}$ if $j=0$, $1$ or $n-1$, respectively. Each {\em flag\/} (maximal totally ordered subset) of $\mathcal{P}$ contains exactly $n+2$ faces, including a unique minimal face $F_{-1}$ (of rank $-1$) and a unique maximal face $F_{n}$ (of rank $n$). Further, we ask that $\mathcal{P}$ be {\em strongly flag-connected\/}, meaning that any two flags $\Phi$ and $\Psi$ of $\mathcal{P}$ can be joined by a sequence of flags $\Phi=\Phi_{0},\Phi_{1},\ldots,\Phi_{l-1},\Phi_{l}=\Psi$, all containing $\Phi\cap\Psi$, such that $\Phi_{i-1}$ and $\Phi_{i}$ are {\em adjacent\/} (differ by exactly one face) for each $i$. Finally, $\po$ satisfies the {\em diamond condition\/}, namely if $F$ is a $(j-1)$-face and $G$ a $(j+1)$-face with $F<G$, then there are exactly two $j$-faces $H$ such that $F<H<G$.

When $F$ and $G$ are two faces of a polytope $\mathcal{P}$ with $F \leq G$, we call $G/F := \{H \mid F \leq H \leq G\}$ a {\em section\/} of $\mathcal{P}$. We usually identify a face $F$ with the section $F/F_{-1}$. For a face $F$ the section $F_{n}/F$ is called the {\em co-face of $\mathcal{P}$ at\/} $F$, or the {\em vertex-figure at\/} $F$ if $F$ is a vertex.  An abstract polytope is {\em simple\/} if all its vertex-figures are simplices. 

An {\em automorphism\/} of a polytope $\po$ is a bijection of the set of faces of $\po$ that preserves the order. The ({\em combinatorial\/}) {\em automorphism group\/} of an abstract polytope $\mathcal{P}$ will be denoted by $\Gamma(\mathcal{P})$.  

\section{Polytopes from properly (edge) colored graphs}
\label{polsfromgraphs}

As before, let $\mathcal{G}$ denote a properly (edge) colored $r$-regular graph with an $r$-edge coloring map $c: E\rightarrow R$ with color set $R$ of cardinality $r$. The main goal of this section is to construct a simple abstract polytope $\mathcal{P}_{\mathcal{G}}$ of rank $r$ from $\mathcal{G}$ in such a way that $\mathcal{G}$ is the $1$-skeleton (edge graph) of $\mathcal{P}_{\mathcal{G}}$. These polytopes are essentially the duals of the colored simplicial complexes or triangulations studied in Bracho \& Montejano~\cite{bracho}.

We begin by defining the face set and the partial order of $\mathcal{P}_{\mathcal{G}}$, and then establish that this  determines an abstract polytope of rank $r$. To begin with, we require the following relations $\sim_C$ on $V$, with $C\subseteq R$, which are easily seen to be equivalence relations. Two vertices $v$ and $w$ in $V$ are said to be {\em $C$-equivalent\/}, or $v\sim_{C} w$ for short, if there exists an edge path of $\mathcal{G}$ from $v$ to $w$ whose edges are colored with colors from $C$. The face structure of the polytope $\mathcal{P}_{\mathcal{G}}$ can now be described as follows.
 
In describing the face structure of $\mathcal{P}_{\mathcal{G}}$. It is convenient to initially suppress the face of rank $-1$ and concentrate entirely on faces of non-negative ranks; the missing $(-1)$-face will be appended at the end.

Accordingly, we define the faces of rank $j=0,1,\ldots,r$ by saying that the vertex set of a typical $j$-face $F$ of $\mathcal{P}_{\mathcal{G}}$ consists of all those vertices of $\mathcal{G}$ that are connected to a given vertex $v$ of $\mathcal{G}$ by an edge path using only colors from a given set $C$ of $j$ colors; that is, the vertex set of $F$ consists of those vertices of $\mathcal{G}$ that are equivalent to $v$ under $\sim_C$. When $C=\emptyset$ this determines the vertices of $\mathcal{P}_\mathcal{G}$ and establishes $V$ as the full vertex set of $\mathcal{P}_\mathcal{G}$; and when $C=R$ we obtain the unique $r$-face $F_r$ (with vertex set $V$).

Thus, a typical $j$-face $F$ of $\mathcal{P}_{\mathcal{G}}$ can be represented as a pair $(C,v)$, where $v$ is a vertex of $\mathcal{G}$ and $C$ is a $j$-subset of $R$. We then write $F=(C,v)$. It will become clear that, once the partial order on the faces of $\mathcal{P}_\mathcal{G}$ has been defined, the vertex set of a face $F$ coincides precisely with the $0$-faces of $\mathcal{P}_\mathcal{G}$ incident to $F$. In general, a face of $\mathcal{P}_\mathcal{G}$ can have many possible designations. In fact, $(C,v)=(D,w)$ if and only $C=D$ and $v$ can be obtained from $w$ (and hence $w$ from $v$) by a path of edges involving only colors from $C$. 

Given a $j$-face $F=(C,v)$ and a $k$-face $G=(D,w)$ of $\mathcal{P}_{\mathcal{G}}$ as above, we say that $F\leq G$ if and only if the vertex set of $F$ is contained in the vertex set of $G$; that is, $F\leq G$ if and only if $C\subseteq D$ (and hence $j\leq k$) and $v$ can be obtained from $w$ by a path involving only edges with colors from $D$. This is consistent with our definition of equality of faces. 

Proceeding with the general discussion, we can immediately make one observation. If $F=(C,v)$ and $G=(D,w)$ are two faces of $\mathcal{P}_{\mathcal{G}}$ of non-negative rank with $F\leq G$, then necessarily $G=(D,v)$; in other words, in designating the larger face we may replace $w$ by $v$. This immediately follows from the definition of the partial order, since $F\leq G$ implies that $v$ can be obtained from $w$ by moving along a path involving only edges with colors from $D$. As an important consequence, any chain of mutually incident faces of $\mathcal{P}_{\mathcal{G}}$ of non-negative rank can be represented in such a way that their second components all coincide with the second component of the smallest face in this chain. In particular, this applies to representations of flags (for now, flags do not contain the $(-1)$-face).

We claim that $\mathcal{P}_{\mathcal{G}}$, partially ordered as described, is an abstract polytope of rank $r$. Clearly, the rank of $\mathcal{P}_{\mathcal{G}}$ is $r$. In fact, every flag is of the form $\Phi:=\{F_{0},F_{1},\ldots,F_r\}$, where $F_{j}=(C_{j},v)$ is a $j$-face for each $j=0,\ldots,r$, with $C_{0}\subset C_{1}\subset \ldots \ldots \subset C_{r}=R$ and $v$ independent of $j$. 

A {\em maximal nested\/} family of subsets of a given finite set is a flag in the Boolean lattice associated with this set. (We use this terminology to avoid confusion with the flags of $\mathcal{P}_{\mathcal{G}}$.) Our previous considerations show that each flag $\Phi$ of $\mathcal{P}_{\mathcal{G}}$ is determined by two parameters, a single vertex $v$ of $\mathcal{G}$ and a maximal nested family $\mathcal{C}:= \{C_{0},C_{1},\ldots,C_{r}\}$ of subsets of the color set $R$. We then write 
\[ \Phi = (\mathcal{C},v) := \{(C_{0},v),(C_{1},v),\ldots,(C_{r},v)\} . \]

It is also immediate that each flag $\Phi=(\mathcal{C},v)$ has exactly one {\em $j$-adjacent\/} flag $\Phi^{j}$ for each $j=0,\ldots,r-1$, differing from $\Phi$ in just the $j$-face. When $j\geq 1$, any flag $j$-adjacent to $\Phi$ can be represented with the same second parameter, $v$, and the same color sets except for the color set at position $j$. Since $C_{j-1}$ is a $(j-1)$-subset of the $(j+1)$-subset $C_{j+1}$, there are exactly two $j$-subsets $C$ such that $C_{j-1}\subset C \subset C_{j+1}$, one given by $C_j$ itself. Hence, there is just one flag $j$-adjacent to $\Phi$, namely the flag determined by the other $j$-subset $C$. This settles the case $j\geq 1$. Now when $j=0$, suppose that $w$ is the vertex of $\mathcal{G}$ adjacent to $v$ by the edge colored with the single color from $C_1$. Then $(\mathcal{C},w)$ is a flag $0$-adjacent to $\Phi$, and this is the only such flag.

It remains to establish the strong flag-connectedness of $\mathcal{P}_{\mathcal{G}}$. Here it is helpful to know the structure of the vertex-figures of $\mathcal{P}_{\mathcal{G}}$. Now, if $(\emptyset,v)$ is any vertex of $\mathcal{P}_{\mathcal{G}}$ and $(C,v)$ and $(D,v)$ are two faces incident with $(\emptyset,v)$, then $(C,v)\leq (D,v)$ in $\mathcal{P}_{\mathcal{G}}$ if and only if $C\subseteq D$. Hence, the vertex-figure of $\mathcal{P}_{\mathcal{G}}$ at $(\emptyset,v)$ is isomorphic to the $(r-1)$-simplex, or equivalently, the Boolean lattice on the color set $R$. Thus, all vertex-figures of $\mathcal{P}_{\mathcal{G}}$ are $(r-1)$-simplices and in particular are flag-connected.

\begin{lemma}
\label{flagconn}
$\mathcal{P}_{\mathcal{G}}$ is strongly flag-connected. 
\end{lemma}

\begin{proof}
We need to establish that, if $\Phi$ and $\Psi$ are two flags of $\mathcal{P}_{\mathcal{G}}$, then there exists a sequence of successively adjacent flags $\Phi=\Phi_0,\Phi_1,\ldots,\Phi_{l-1},\Phi_l=\Psi$ of $\mathcal{P}_{\mathcal{G}}$ such that $\Phi\cap\Psi\subseteq\Phi_i$ for all $i=0,\ldots,l$. 

Let $\Phi = (\mathcal{C},v)$ and $\Psi = (\mathcal{D},w)$ be two flag of $\mathcal{P}_\mathcal{G}$, where $\mathcal{C}=\{C_{0},\ldots,C_r\}$ and $\mathcal{D}=\{D_{0},\ldots,D_r\}$ are two maximal nested families of subsets of $R$. Define 
\[J:=\{j \mid (C_{j},v)=(D_{j},w)\},\]
and observe that $J\neq\emptyset$ since $r\in J$. Let $m$ denote the smallest suffix in $J$. 

We now exploit the connectedness of $\mathcal{G}$ to construct the sequence of flags. Since $m\in J$, we know that $C_{m}=D_{m}$ and that $v$ can be joined to $w$ by a path of length $p$ (say) involving only edges of $\mathcal{G}$  with colors from $C_m$. If $m=0$, then necessarily $v=w$ and $p=0$. If $m=1$, then only the edge colored with the color in $C_1$ is involved in the path and we may take $p=1$. 

The proof proceeds by induction on $p$. When $p=0$, the vertices $v$ and $w$ of $\mathcal{G}$ coincide, so we can simply appeal to the strong flag-connectedness of the vertex-figure of $\mathcal{P}_\mathcal{G}$ at $(\emptyset,v)$ to find the desired sequence of flags. Now, let $p\geq 1$. Then necessarily $m\geq 1$ and $C_{m}\neq \emptyset$. Suppose 
\[ v=v_{0},v_{1},\ldots, v_{p-1},v_{p}=w \]
is the sequence of successively adjacent vertices in a path from $v$ to $w$ in $\mathcal{G}$ all of whose edges are colored with colors from $C_m$. The key inductive step in the proof is to first join the given flag $\Phi=(\mathcal{C},v)$, by a suitable sequence of successively adjacent flags, to a new flag $\Lambda=(\mathcal{E},x)$ associated with the vertex $x:=v_{1}$ and with a maximal nested family $\mathcal{E}$ of subsets of $R$. Once this has been accomplished, we then further extend this first sequence by a suitable second sequence of successively adjacent flags joining $\Lambda$ and $\Psi$; the existence of this second sequence is guaranteed by the inductive hypothesis for $p-1$, applied to $\Lambda$ and $\Psi$. Finally, when concatenated, these two sequences yield the desired sequence joining $\Phi$ to $\Psi$. 

We begin by constructing $\Lambda$. Let $c_{1}$ denote the color of the edge joining $v$ to $v_{1}=x$. Then notice that in $\mathcal{P}_{\mathcal{G}}$ we have
\begin{equation}
\label{obs}
(C_{0},v), (D_{0},x) \,\leq\, (\{c_{1}\},v) \leq (C_{m},v) = (D_{m},x),  
\end{equation}
since $c_{1}\in C_{m}=D_{m}$ (recall the definition of $m$) and $x$ and $v$ are joined by an edge colored $c_1$. We now proceed in three steps. First, by (\ref{obs}) we can find a flag $\Lambda'$ (say) which has $(C_{0},v)$ as its $0$-face and $(\{c_{1}\},v)$ as its $1$-face and also includes the entire set $\Phi\cap\Psi$; when $m=1$ we can take $\Lambda'=\Phi$. Then, since $\Phi$ and $\Lambda'$ share a vertex and the vertex-figures are strongly flag-connected, we can join $\Phi$ and $\Lambda'$ by a sequence of successively adjacent flags, all containing $\Phi\cap\Lambda'$ and hence also $\Phi\cap\Psi$. Second, replace the $0$-face of $\Lambda'$ by $(D_{0},x)$ to obtain a new flag $\Lambda''$ which is $0$-adjacent to $\Lambda'$. Then extend the already existing sequence of flags joining $\Phi$ to $\Lambda'$ by this new flag $\Lambda''$. Third, take any maximal nested family $\mathcal{E} = \{E_{0},\ldots,E_{r}\}$ of subsets of $R$ with $E_{1}=\{c_1\}$ and $E_{j}=C_{j}=D_{j}$ for each $j\in J$, and define
\[ \Lambda := (\mathcal{E},x) = \{ (E_{0},x), \ldots, (E_{r},x) \}.\]
Note that a maximal nested family $\mathcal{E}$ of this kind exists since $c_{1}\in C_{m}$ and $C_m \subseteq C_j=D_j$ for each $j\in J$. Then it is immediately clear that $\Lambda$ shares a vertex with $\Lambda''$ and also contains $\Phi\cap\Lambda$; here the latter follows from the fact that $c_{1}\in C_{j}=D_{j}$ for each $j\in J$. Now we can complete the argument by appealing once more to the flag-connectedness of the vertex-figures. In fact, since $\Lambda$ and $\Lambda''$ have a vertex in common, we can find a sequence of successively adjacent flags from $\Lambda''$ to $\Lambda$, which all contain $\Lambda''\cap\Lambda$ and hence also $\Phi\cap\Psi$. We now have connected $\Phi$ and $\Lambda$ by a flag sequence, via $\Lambda'$ and $\Lambda''$. 
\end{proof}

It is immediately clear from the construction that the $1$-skeleton of $\mathcal{P}_{\mathcal{G}}$ is isomorphic to $\mathcal{G}$. In our subsequent discussion we usually identify the $1$-skeleton of $\mathcal{P}_{\mathcal{G}}$ with $\mathcal{G}$ and simply denote its vertices $(\emptyset,v)$ and edges $\{(\emptyset,v),(\emptyset,v')\}$ by $v$ or $\{v,v'\}$, respectively. Notice also that the vertex-figures of $\mathcal{P}_\mathcal{G}$ are $(r-1)$-simplices, so in particular each vertex of $\mathcal{P}_{\mathcal{G}}$ is contained in $r!$ flags. 

Each facet of $\mathcal{P}_{\mathcal{G}}$ is of the form $(R\setminus\{b\},v)$, where $b$ is a color in $R$ and $v$ is a vertex of $\mathcal{G}$. The pair $(b,v)$ determines a connected properly (edge) colored $(r-1)$-regular graph ${\mathcal{G}_{b,v}}$ with color set $R\setminus\{b\}$, namely the connected component of $v$ in the new graph $\mathcal{G}_b$ obtained from $\mathcal{G}$ by deleting all the edges of $\mathcal{G}$ colored $b$. Then the facet $(R\setminus\{b\},v)$ of $\mathcal{P}_{\mathcal{G}}$ is isomorphic to the colorful polytope $\mathcal{P}_{\mathcal{G}_{b,v}}$ of rank $r-1$ associated with this graph ${\mathcal{G}_{b,v}}$. Conversely, given any color $b$ of $R$ and any vertex $v$ of $\mathcal{G}$, the corresponding connected component of $v$ in $\mathcal{G}_{b}$ gives rise to a facet of $\mathcal{P}_{\mathcal{G}}$ of the form $(R\setminus\{b\},v)$ in this manner. 

In summary, we have established the following theorem.

\begin{theorem}
\label{mainthm}
Let $\mathcal{G}$ be a connected properly (edge) colored $r$-regular graph. Then $\mathcal{P}_{\mathcal{G}}$ is an abstract polytope of rank $r$, called the {\em colorful polytope associated with $\mathcal{G}$}.  In particular, $\mathcal{P}_{\mathcal{G}}$ is a simple polytope with $1$-skeleton isomorphic to $\mathcal{G}$. The facets of $\mathcal{P}_{\mathcal{G}}$ are in one-to-one correspondence with the colorful polytopes of rank $r-1$ associated 
with the connected components of the graph obtained from $\mathcal{G}$ by the deletion of the edges in a single color class.
\end{theorem}

Figures~\ref{K4} and ~\ref{Botella} illustrate two small examples of colorful polyhedra (polytopes of rank $3$). The first is the hemi-cube $\{4,3\}/2$, a regular map on the projective plane obtained from the properly (edge) colored $3$-regular graph given  by the (unique) $1$-factorization of $K_4$.  Note here that the colorful polyhedron $\mathcal{P}_{K_4} = \{4,3\}/2$ is not (isomorphic to) a convex polyhedron, although the underlying graph $K_4$ does of course occur as the 1-skeleton of a convex polyhedron, namely the tetrahedron. However, these two polyhedra are closely related, in that one is the Petrie dual of the other (see \cite{McMS02}). The second is a map with four $2$-faces (two squares and two octagons) on the Klein bottle and is associated with the $3$-regular graph on $8$ vertices that is shown on the left in Figure~\ref{Botella}.

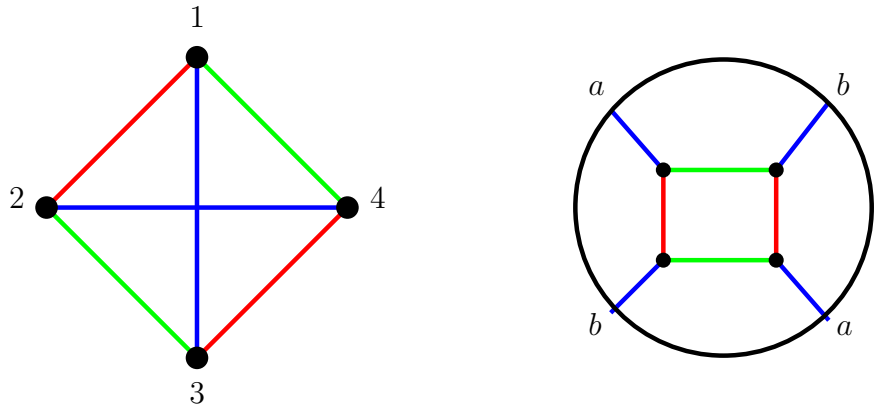
\begin{figure}[p]
\begin{center}
\psset{unit=1.0cm, linewidth=0.06cm}
   \begin{pspicture}(0,0)(7.8,8.3)

  \psline [linecolor=green](0,3)(-2,5)
   \psline [linecolor=green](0,7)(2,5)
   \psline [linecolor=red](0,3)(2,5)
   \psline [linecolor=red](-2,5)(0,7)
    \psline [linecolor=blue](-2,5)(2,5)
    \psline [linecolor=blue](0,3)(0,7)
   
\put(8.5,6.5){$b$}
\put(5.2,6.5){$a$}
\put(5.2,3.3){$b$}
\put(8.5,3.3){$a$}

 \cnode*(0,3){0.15}{00}\put(-0.1,2.4){$3$}
  \cnode*(-2,5){0.15}{10}\put(-2.5,5){$2$}
  \cnode*(0,7){0.15}{20}\put(-0.1,7.4){$1$}
   \cnode*(2,5){0.15}{30}\put(2.3,5){$4$}

\psline [linecolor=blue](7.7,5.5)(8.4,6.4)
\psline [linecolor=blue](6.2,5.5)(5.5,6.3)
\psline [linecolor=blue](6.2,4.3)(5.5,3.6)
\psline [linecolor=blue](7.7,4.3)(8.4,3.5)
\psline [linecolor=green](6.2,5.5)(7.7,5.5)
\psline [linecolor=green](6.2,4.3)(7.7,4.3)
\psline [linecolor=red](6.2,5.5)(6.2,4.3)
\psline [linecolor=red](7.7,4.3)(7.7,5.5)

\pscircle(7,5){2}

\cnode*(7.7,5.5){0.10}{40}
\cnode*(6.2,5.5){0.10}{50}
\cnode*(6.2,4.3){0.10}{60}
\cnode*(7.7,4.3){0.10}{70}

    \end{pspicture}
\vspace{-1.5cm}
\caption{\label{K4} The hemi-cube derived from $K_4$.}

\end{center}
\end{figure}

\begin{figure}[p]
\begin{center}
\psset{unit=1.0cm, linewidth=0.06cm}
  \begin{pspicture}(0,0)(7.8,8.3)
 
   \psline [linecolor=green](-4,6)(-2,7)
   \psline [linecolor=green](-2,5)(0,6)
   \psline [linecolor=red](-4,6)(0,6)
   \psline [linecolor=blue](-4,6)(-2,5)
   \psline [linecolor=blue](0,6)(-2,7)

   \psline [linecolor=green](-4,1)(-2,2)
   \psline [linecolor=green](-2,0)(0,1)
   \psline [linecolor=red](-4,1)(0,1)
   \psline [linecolor=blue](-4,1)(-2,0)
   \psline [linecolor=blue](0,1)(-2,2)
   \psline [linecolor=red](-2,5)(-2,2)
     \pscurve  [linecolor=red](-2,7)(1,6)(1,1)(-2,0)
     
\cnode*(-2,5){0.15}{00}\put(-1.9,4.5){$3$}
\cnode*(-4,6){0.15}{10}\put(-4.5,6){$2$}
\cnode*(-2,7){0.15}{20}\put(-1.9,7.4){$1$}
 \cnode*(0,6){0.15}{30}\put(0.3,6){$4$}
        
  \cnode*(-2,0){0.15}{00}\put(-1.9,-0.5){$3'$}
  \cnode*(-4,1){0.15}{10}\put(-4.5,1){$2'$}
  \cnode*(-2,2){0.15}{20}\put(-1.9,2.2){$1'$}
   \cnode*(0,1){0.15}{30}\put(0.3,1){$4'$}
     
\psline(4,6)(5,6) 
\psline [linecolor=red](5,6)(7,6)
\psline{->>}(7,6)(8,6) 
\psline(7.9,6)(9,6)
\psline [linecolor=red](9,6)(11,6)
\psline(11,6)(12,6) 
\psline(4,6)(4,5.4) 
\psline{<-}(4,5.6)(4,5) 
\psline [linecolor=red](4,5)(4,3)
\psline(4,3)(4,2.4) 
\psline{<-}(4,2.5)(4,2) 
\psline(4,2)(5,2) 
\psline [linecolor=red](5,2)(7,2)
\psline(7,2)(7.8,2) 
\psline{<<-}(7.5,2)(9,2)
\psline [linecolor=red](9,2)(11,2)
\psline(11,2)(12,2) 
\psline(12,6)(12,5.5) 
\psline{->}(12,5)(12,5.6) 
\psline [linecolor=red](12,5)(12,3)
\psline(12,3)(12,2.4) 
\psline{<-}(12,2.5)(12,2) 
\psline [linecolor=red](8,5)(8,3)
\psline [linecolor=blue](5,6)(4,5)
\psline [linecolor=blue](4,3)(5,2)
\psline [linecolor=blue](7,2)(8,3)
\psline [linecolor=blue](8,5)(7,6)
\psline [linecolor=green](8,5)(9,6)
\psline [linecolor=green](11,6)(12,5)
\psline [linecolor=green](8,3)(9,2)
\psline [linecolor=green](11,2)(12,3)
\cnode*(4,6){0.05}{40}  
\cnode*(5,6){0.15}{50}\put(5,6.3){$1'$}   
\cnode*(7,6){0.15}{60}\put(7,6.3){$3$}   
\cnode*(9,6){0.15}{70}\put(9,6.3){$1$}   
\cnode*(11,6){0.15}{80}\put(11,6.3){$3'$}   
\cnode*(12,6){0.05}{90}
\cnode*(4,5){0.15}{60}\put(3.5,4.8){$4'$}  
\cnode*(4,3){0.15}{60}\put(3.5,2.8){$2'$}  
\cnode*(4,2){0.05}{100}
\cnode*(5,2){0.15}{50}\put(5,1.5){$3'$}  
\cnode*(7,2){0.15}{80}\put(7,1.5){$1$}  
\cnode*(9,2){0.15}{70}\put(9,1.5){$3$}   
\cnode*(11,2){0.15}{80}\put(11,1.5){$1'$}  
\cnode*(12,2){0.05}{90} 
\cnode*(12,5){0.15}{60}\put(12.3,5){$4'$}  
\cnode*(12,3){0.15}{70}\put(12.3,3){$2'$}  
\cnode*(8,5){0.15}{60}\put(8.3,4.8){$2$}  
\cnode*(8,3){0.15}{60}\put(8.3,2.8){$4$}  
   
    \end{pspicture}
    \vspace{.4cm}
\caption{\label{Botella} A $3$-regular graph and its polyhedron on the Klein bottle.}

\end{center}
\end{figure}
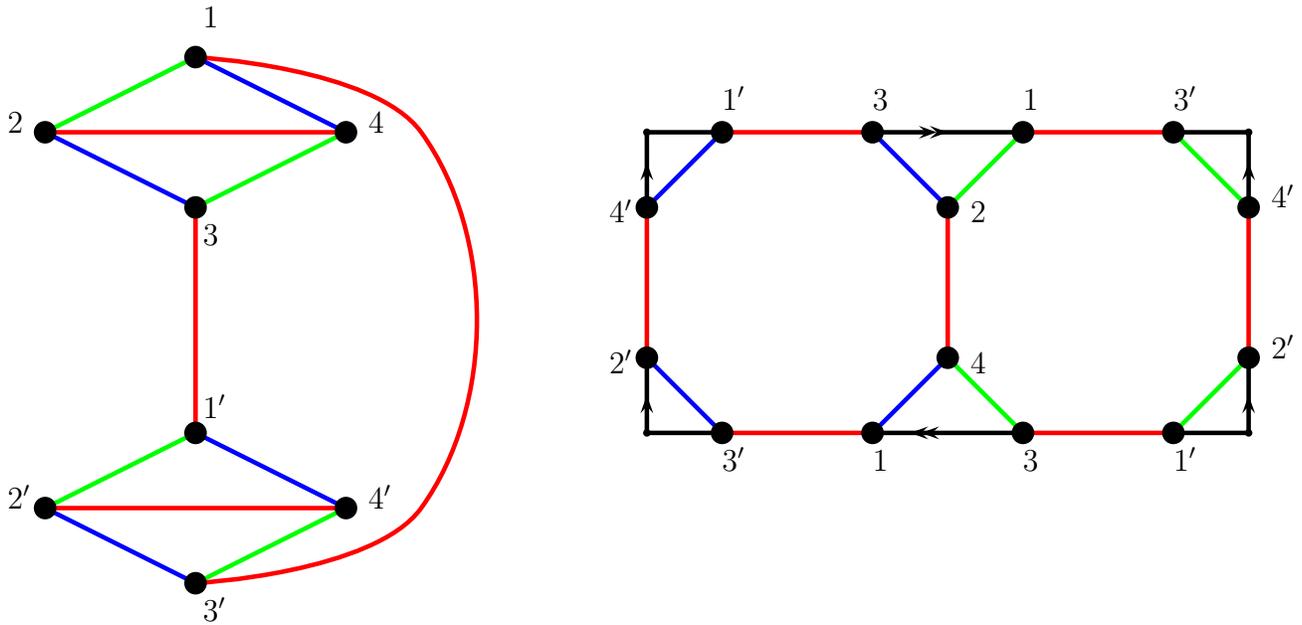

It is interesting to explore the possibility of when a colorful polytope is isomorphic to a (simple) convex polytope. By construction, the $2$-faces of any colorful polytope $\mathcal{P}_{\mathcal{G}}$ exhibit an alternating pattern of two colors on their edges and hence must have an even number of edges; in fact, every $2$-face is of the form $(C,v)$ where $v$ is a vertex of the graph $\mathcal{G}$ and $C$ is a set of two colors from $R$. This evenness condition is violated when $K_4$ is realized as the edge graph of the tetrahedron. On the other hand, we know that any two simple convex $r$-polytopes with the same $1$-skeleton are isomorphic (see \cite{BlindMan87,kalai}). Hence, given a properly (edge) colored $r$-regular graph $\mathcal{G}$ that is the $1$-skeleton of a convex $r$-polytope $\mathcal{Q}$, then, if the corresponding colorful $r$-polytope $\mathcal{P}_{\mathcal{G}}$ is also a convex polytope, these two polytopes $\mathcal{P}_{\mathcal{G}}$ and $\mathcal{Q}$ must be isomorphic. A natural question to ask here is when the edges of the 1-skeleton of a simple convex polytope can be colored is such a way that the corresponding colorful polytope is again a convex polytope (and hence isomorphic to the given polytope).

Note that, in general, $r$-regular graphs that admit more than one 1-factorization will give rise to more than one colorful polytope. For example, Figure~\ref{cubes} shows two edge colorations of the edge graph of the cube that give rise to two different colorful polyhedra. In fact, one polyhedron is the cube itself (on the $2$-sphere), while the other is a map on the torus with two octagons and two squares as 2-faces. The first polyhedron is combinatorially regular, with an automorphism group acting flag-transitively, while the other has an automorphism group with three flag orbits.

\begin{figure}[h]
\begin{center}
\vspace{-1cm}
\includegraphics[width=15cm]{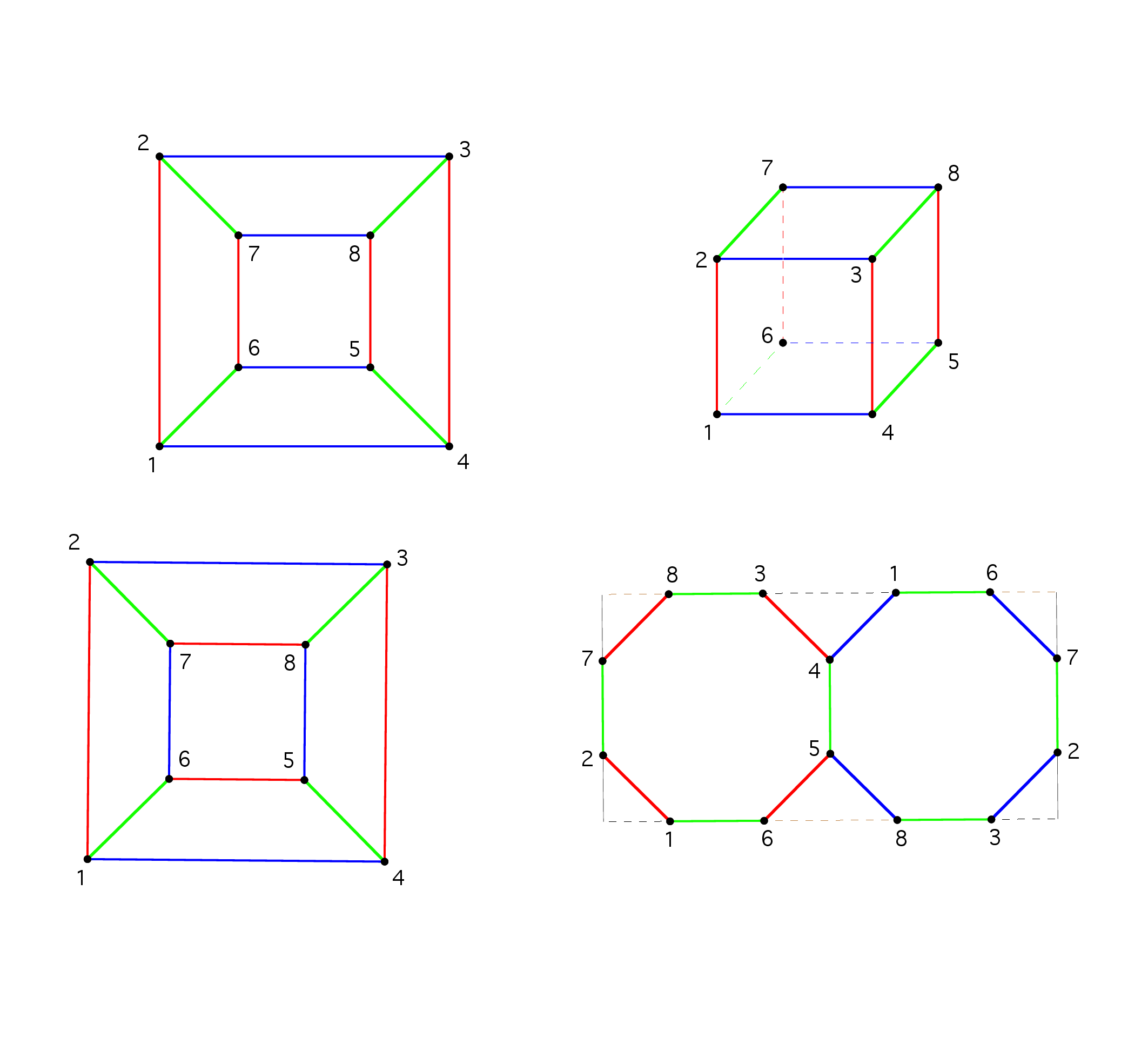} 
\vspace{-3cm}
\caption{Two edge colorings of the cube and their colorful polytopes.}
\label{cubes}
\end{center}
\end{figure}

\section{The automorphism group of a colorful polytope}
\label{autgrcol}

As before, let $\mathcal{G}$ be a properly (edge) colored $r$-regular graph with edge coloring map $c: E \rightarrow R$ and graph automorphism group $\Gamma(\mathcal{G})$, and let $\mathcal{P}_{\mathcal{G}}$ be the corresponding colorful polytope of rank $r$. In this section we establish that the combinatorial automorphism group $\Gamma(\mathcal{P}_{\mathcal{G}})$ of the polytope $\mathcal{P}_{\mathcal{G}}$ is given by the subgroup $\Gamma_{c}(\mathcal{G})$ of $\Gamma(\mathcal{G})$ of all color respecting automorphisms of $\mathcal{G}$.

There is a natural homomorphism 
\[ \kappa: \Gamma_{c}(\mathcal{G}) \to S_R \] 
from $\Gamma_{c}(\mathcal{G})$ to the symmetric group $S_R$ on the color set $R$. This associates with every $\gamma \in \Gamma_{c}(\mathcal{G})$ the permutation $\bar{\gamma}:=\kappa(\gamma)$ of the colors induced by $\gamma$. More explicitly, if $c'$ is any color in $R$, then $\bar{\gamma}(c'):=c(\gamma(e))$, where $e$ is any edge of $\mathcal{G}$ with $c(e)=c'$; thus 
\[ \bar{\gamma}(c(e)):=c(\gamma(e))\;\;\,(e\in E).\] 
Moreover, there is also a natural injective homomorphism 
\[ \mu: \Gamma_{c}(\mathcal{G})\to\Gamma(\mathcal{P}_{\mathcal{G}}) \]
from $\Gamma_{c}(\mathcal{G})$ to $\Gamma(\mathcal{P}_{\mathcal{G}})$, which allows us to identify $\Gamma_{c}(\mathcal{G})$ with a subgroup of $\Gamma(\mathcal{P}_{\mathcal{G}})$. In fact, every $\gamma \in \Gamma_{c}(\mathcal{G})$ induces a (well-defined) polytope automorphism $\widehat{\gamma}:=\mu(\gamma)$ of $\mathcal{P}_{\mathcal{G}}$ defined by           
\begin{equation}
\label{phihat}
\widehat{\gamma}((C,v)):=(\bar{\gamma}(C),\gamma(v)).
\end{equation}
Note that $\widehat{\gamma}$ is trivial if and only of $\gamma$ is trivial (choose $C=\emptyset$). Our main objective is to show that  $\Gamma_{c}(\mathcal{G})$ is in fact the full automorphism group $\Gamma(\mathcal{P}_{\mathcal{G}})$ of the colorful polytope.  

\begin{theorem}
\label{automor}
Let $\mathcal{G}$ be a properly (edge) colored $r$-regular graph, and let $\mathcal{P}_\mathcal{G}$ be the corresponding colorful polytope. Then 
$\Gamma(\mathcal{P}_{\mathcal{G}}) = \Gamma_{c}(\mathcal{G})$.
\end{theorem}

\begin{proof}
We already know that $\Gamma_{c}(\mathcal{G})$ can be viewed as a subgroup of $\Gamma(\mathcal{P}_{\mathcal{G}})$. It remains to show that this subgroup is the full group $\Gamma(\mathcal{P}_{\mathcal{G}})$. To this end, consider the homomorphism 
\[ \pi: \Gamma(\mathcal{P}_{\mathcal{G}}) \to \Gamma(\mathcal{G})\] 
defined by $\pi(\alpha):=\alpha |_{\mathcal{G}}$, the latter being the restriction of $\alpha$ to the $1$-skeleton $\mathcal{G}$ of $\mathcal{P}_{\mathcal{G}}$. 

We first show that $ker(\pi )$ is trivial, so $\pi$ is injective. If $\alpha \in ker(\pi)$, then $\alpha$ fixes every vertex and every edge of $\mathcal{G}$ and therefore also of $\mathcal{P}_{\mathcal{G}}$; moreover, $\pi(\alpha)$ preserves the colors of the edges of $\mathcal{G}$ and hence lies in $\Gamma_{c}(\mathcal{G})$ (in fact, $\overline{\pi(\alpha)}$ is the identity permutation on $R$). It follows that $\alpha$ must fix every face $(K,v)$ of $\mathcal{P}_{\mathcal{G}}$, since its vertex set is (even pointwise) invariant under $\alpha$; in fact, each vertex and each edge, as well as each edge color, of the properly (edge) colored graph $\mathcal{G}$ are fixed by $\alpha$ or $\overline{\pi(\alpha)}$, respectively. Bear in mind here that a face of $\mathcal{P}_{\mathcal{G}}$ is completely determined by its vertex set. Hence $\alpha$ is trivial.

Thus, via $\pi$, we can identify $\Gamma(\mathcal{P}_{\mathcal{G}})$ with a subgroup of $\Gamma(\mathcal{G})$. It remains to show that $\Gamma(\mathcal{P}_{\mathcal{G}})$ actually lies in $\Gamma_{c}(\mathcal{G})$; that is, each polytope automorphism of $\mathcal{P}_{\mathcal{G}}$ corresponds, under $\pi$, to a color respecting graph automorphism of $\mathcal{G}$.

Let $\phi \in \Gamma(\mathcal{P}_{\mathcal{G}})$. We need to show that if $e,e'$ are two edges of $\mathcal{G}$ of the same color, then their images under $\phi$ are again two edges of the same color (possibly distinct from the color on $e$ and $e'$). So let $e$ and $e'$ be edges of $\mathcal{G}$ with $c(e)=c(e')=:c'$. We wish to prove that then also $c(\phi(e))=c(\phi(e'))$. Since $\mathcal{G}$ is a connected graph, there exists a sequence  
\begin{equation}
\label{epath}
e=e_0,v_1,e_1,v_2,\dots,v_{m-1},e_{m-1},v_m,e_m=e' 
\end{equation} 
of successively incident edges and vertices of $\mathcal{G}$ that forms a path connecting $e$ and $e'$. Then the images under $\phi$ also form a sequence
\begin{equation}
\label{phiepath}
\phi(e)\!=\!\phi(e_0),\phi(v_1),\phi(e_1),\phi(v_2),\dots,\phi(v_{m-1}),\phi(e_{m-1}),\phi(v_m),\phi(e_m)\!=\!\phi(e'),
\end{equation}
of successively incident edges and vertices of $\mathcal{G}$, now forming a path from $\phi(e)$ and $\phi(e')$. We first settle the particular case when none of the intermediate edges $e_1,\ldots,e_{m-1}$ of the sequence in (\ref{epath}) has the color $c'$. The general case can be reduced to this special case.

Suppose none of the edges $e_1,\ldots,e_{m-1}$ has the color $c'$. Set $C':=R\setminus \{c'\}$, where as before $R$ is the color set of $\mathcal{G}$. Then $F:=(C',v_1)$ is a facet of $\mathcal{P}_{\mathcal{G}}$ containing all the vertices $v_1,\ldots,v_m$ (as vertices) and all the intermediate edges $e_1,\ldots,e_{m-1}$ (as edges), but not the first edge $e$ and last edge $e'$, of the sequence. In fact, $C'$ has cardinality $r-1$ and the intermediate sequence $v_1,e_1,v_2,\dots,v_{m-1},e_{m-1},v_m$ defines an edge path in $\mathcal{G}$ consisting only of edges colored with colors from $C'$; hence, since $F$ is the connected component in $\mathcal{G}$ under the equivalence relation $\sim_{C'}$, the members of the intermediate sequence must all belong to $F$, either as vertices or as edges. On the other hand, $e$ and $e'$ have color $c'$ and hence do not belong to $F$. Notice also that $F=(C',v_j)$ for each $j=1,\ldots,m$. 

Now recall that the vertex-figures of $\mathcal{P}_{\mathcal{G}}$ are $(r-1)$-simplices. Hence, if $v$ is any vertex and $G$ any facet of $\mathcal{P}_{\mathcal{G}}$ containing $v$, then there is just one edge of $\mathcal{P}_{\mathcal{G}}$ with vertex $v$ that does not lie $G$. In particular, the color of this edge is just the one color that is not being used to define the facet $G$. 

Now we can argue as follows. Since all the vertices and all the intermediate edges of the sequence in (\ref{epath}) lie in the  facet $F$ of $\mathcal{P}_{\mathcal{G}}$ and $\phi$ is a polytope automorphism, all the vertices and all the intermediate edges of the sequence in (\ref{phiepath}) must lie in the facet $\phi(F)$ of $\mathcal{P}_{\mathcal{G}}$. Similarly, the first edge $\phi(e)$ and the last edge $\phi(e')$ in (\ref{epath}) are not edges of $\phi(F)$. Hence, if our previous observations are applied with $G:=\phi(F)$ and $v:=\phi(v_1)$ or $v:=\phi(v_m)$, we immediately see that the edges $\phi(e)$ and $\phi(e')$ must be colored with the color not being used on $\phi(F)$.  More explicitly, if $\phi(F)=(R\setminus\{c''\},\phi(v_1))$ (say), then $c(\phi(e))=c(\phi(e'))=c''$, as desired. This concludes the proof in the special case when none of the edges $e_1,\ldots,e_{m-1}$ in (\ref{epath}) has the color $c'$.

Finally, in the general case we can simply view the sequence in (\ref{epath}) as a concatenation of subsequences in which only the first edge and the last edge have the same color $c'$ as $e$ and $e'$ (in particular, none of the intermediate edges in a subsequence has color $c'$). In fact, if 
\[ e=e_{0}=:e_{j_0}, e_{j_1},\ldots,e_{j_{k-1}},e_{j_k}:=e_{m}=e', \]
with $0=j_{0}<j_{1}< \ldots < j_{k-1}<j_{k}=m$, are all the edges of the sequence (\ref{epath}) which have color $c'$, then 
for every $l=0,\ldots,k-1$ the subsequence
\[ e_{j_l},v_{j_{l}+1},e_{j_{l}+1},v_{j_{l}+2}, \ldots, v_{j_{l+1}},e_{j_{l+1}} \]
is of the required type and hence $c(\phi(e_{j_l}))=c(\phi(e_{j_{l+1}}))$. Thus $c(\phi(e))=c(\phi(e'))$, the desired conclusion.
\end{proof}

Recall that an (abstract) polytope is said to be {\em regular} if its automorphism groups acts transitively on the flags. A  colorful polytope $\mathcal{P}_{\mathcal{G}}$ (in fact, any simple $r$-polytope) is regular if and only if its automorphism group acts transitively on the vertices and the stabilizer of a vertex acts as a symmetric group $S_r$ on the $r$ edges that contain the vertex. In particular, if $\mathcal{P}_{\mathcal{G}}$ is regular, then the natural homomorphism $\kappa: \Gamma(\mathcal{P}_{\mathcal{G}})=\Gamma_{c}(\mathcal{G}) \to S_R$ defined earlier is surjective, so every permutation of colors in $R$ can be realized by an automorphism of $\mathcal{G}$ (and $\mathcal{P}$). It would be interesting to know which properly (edge) colored graphs have maximum possible symmetry of this kind.  

\section{The flag-adjacency polytope}
\label{flagpo}

Every abstract polytope naturally gives rise to a colorful polytope, its flag-adjacency polytope, that is closely related to its  monodromy polytope describe in Section~\ref{mono}.

Let $\mathcal{P}$ be an abstract $n$-polytope. The {\em flag adjacency graph}, or simply {\em flag graph\/}, $\mathcal{G}$ of $\mathcal{P}$ is the properly (edge) colored $n$-regular graph whose vertices are the flags of $\po$ and whose edges with color $i$ join two vertices if and only if the corresponding flags are {\em $i$-adjacent\/} (differ exactly in their faces of rank $i$); here the underlying set of colors is $R:=\{0,\ldots,n-1\}$. Thus the vertex set $V$ of the flag-graph $\mathcal{G}$ is the set $\mathcal{F}(\mathcal{P})$ of flags of $\mathcal{P}$, and $\mathcal{G}$ itself is just the (properly (edge) colored) $1$-skeleton of the dual of the order complex of $\po$ (see \cite[Section 2C]{McMS02}). From Theorem~\ref{mainthm} we also know that $\mathcal{G}$ is the 1-skeleton of the corresponding colorful polytope $\po_{\mathcal{G}}$ of rank $n$, which we have named the {\em flag adjacency polytope\/}. 

Again let $\mathcal{P}$ be an $n$-polytope, and let $\mathcal{G}$ be its flag graph. By Theorem~\ref{automor}, the automorphism group $\Gamma(\po_{\mathcal{G}})$ of the colorful polytope $\po_{\mathcal{G}}$ for $\mathcal{G}$ is isomorphic to the group $\Gamma_{c}(\mathcal{G})$ of color respecting automorphisms of $\mathcal{G}$. In particular, 
the automorphism group $\Gamma(\po)$ of the original polytope $\po$ can be viewed as a subgroup of $\Gamma_{p}(\mathcal{G})$ and hence of $\Gamma(\po_{\mathcal{G}})$, since every automorphism of $\po$ induces (faithfully) an adjacency preserving bijection of $\fl(\po)$ and hence a color preserving automorphism of $\mathcal{G}$. 
Conversely, a color preserving automorphism of $\mathcal{G}$ is a bijection of its vertex set that preserves the colors of edges. In other words, a color preserving automorphism of $\mathcal{G}$ is a bijection of $\mathcal{F}(\mathcal{P})$ that preserves $i$-adjacency of flags for each $i$, and hence comes from an automorphism of $\po$. Thus $\Gamma_{p}(\mathcal{G})=\Gamma(\po)$.

The next theorem determines the structure of $\Gamma(\po_{\mathcal{G}})$ under mild conditions on $\po$. The {\em $(i,i+1)$-face layer graph\/} of an $n$-polytope $\po$ is the bipartite graph whose vertices are the faces of $\po$ of ranks $i$ and $i+1$ and whose edges represent incidence in $\po$ (the medial layer graphs discussed in Monson \& Weiss~\cite{medlay,mow} and Monson, Pisanski, Schulte \& Weiss~\cite{semsym} are examples of face layer graphs). Recall that a {\em duality\/} of a self-dual $n$-polytope $\po$ is a bijection of $\po$ that reverses incidences between faces. Every duality induces a bijection of $\mathcal{F}(\mathcal{P})$ that sends $i$-adjacent flags to $(n-i-1)$-adjacent flags for each $i$, and therefore determines a color respecting automorphism of the flag graph $\mathcal{G}$. By $\bar{\Gamma}(\po)$ we denote the group of all automorphisms and dualities of~$\po$. This can be viewed as a subgroup of $\Gamma(\po_{\mathcal{G}})$. Clearly, $\bar{\Gamma}(\po) = \Gamma(\po)$ if $\po$ is not self-dual.

\begin{theorem}
Let $\po$ be an $n$-polytope such that none of its $(i,i+1)$-face layer graphs, with $i=0,\ldots,n-2$, is a complete bipartite graph. Let $\po_{\mathcal{G}}$ be the colorful polytope arising from the flag-adjacency graph $\mathcal{G}$ of $\po$. Then $\Gamma(\po_{\mathcal{G}}) = \bar{\Gamma}(\po)$. 
\end{theorem}

\begin{proof}
We shall consider the $2$-faces of $\po_{\mathcal{G}}$. Each $2$-face is given by a pair $(\{i,j\},\Phi)$, where $\{i, j\}$ is a $2$-subset of $R=\{0,\dots n-1\}$ and $\Phi$ is a vertex of $\mathcal{G}$ (that is, a flag of $\po$). Suppose the 2-face $(\{i, j\},\Phi)$ is a $q_{ij}^{(\Phi)}$-gon; necessarily, $q_{ij}^{(\Phi)}$ is even. Clearly, if $i$ and $j$ are non-adjacent, then $q_{ij}^{(\Phi)}=4$ for all vertices $\Phi$ of $\mathcal{G}$. However, by our assumptions on $\po$, if $i$ and $j$ are adjacent, then there is a vertex $\Phi$ of $\mathcal{G}$ with $q_{ij}^{(\Phi)}>4$. 

Now let $\gamma \in \Gamma(\po_{\cal G})=\Gamma_c({\cal G})$, and let $\bar{\gamma}$ be the permutation of $R$ associated with $\gamma$ considered as a color respecting automorphism of $\mathcal{G}$. When viewed as a polytope automorphism, $\gamma$ maps $2$-faces to $2$-faces and face layer graphs to face layer graphs. In particular, if $\{i, j\}$ is a $2$-subset of $R$, then $\gamma$ takes a typical $2$-face $(\{i,j\},\Phi)$ of $\po_{\cal G}$ to $(\{\bar{\gamma}(i),\bar{\gamma}(j)\},\gamma(\Phi))$. We claim that if $i$ and $j$ are non-adjacent then necessarily $\bar{\gamma}(i)$ and $\bar{\gamma}(j)$ are non-adjacent. This follows from our assumption on the face layer graphs of $\po$. In fact, if $\bar{\gamma}(j)=\bar{\gamma}(i)\pm 1$, then the polytope automorphism $\gamma^{-1}$ of $\po_{\mathcal{G}}$ takes a $2$-face $(\{\bar{\gamma}(i),\bar{\gamma}(j)\},\Psi)$ of $\po_{\mathcal{G}}$ with $q_{\,\bar{\gamma}(i)\bar{\gamma}(j)}^{(\Psi)}>4$ to the $2$-face $(\{i,j\},\gamma^{-1}(\Psi))$ with $q_{ij}^{(\gamma^{-1}(\Psi))}=4$ (bear in mind that $i$ and $j$ are non-adjacent). The existence of a suitable vertex $\Psi$ of $\mathcal{G}$ is guaranteed since the face layer graph of $\po$ corresponding to the adjacent ranks $\bar{\gamma}(i),\bar{\gamma}(j)$ is not a complete bipartite graph.

Now we are nearly done. In fact, since $0$ and $n-1$ are distinguished among the colors by having only one ``adjacent'' color, $\bar{\gamma}$ must either fix or interchange them. Hence $\bar{\gamma}$ must either be the identity map on $R$ or interchange the colors $i$ and $n-1-i$ in $R$ for each $i$; accordingly, $\gamma$ is an automorphism or a duality of the underlying polytope $\po$. 
\end{proof}

In rank $3$, the construction of colorful polytopes (polyhedra) is related to some well-known operations on maps. It is not difficult to see that the colorful polyhedron arising from the flag-adjacency graph of an abstract polyhedron is isomorphic to the truncation of the medial of the original polyhedron (see Hubard, Orbanic \& Weiss~\cite{d-auto} and Orbanic, Pellicer \& Weiss~\cite{korbit} for basic definitions and results about medials and truncations). Alternatively, as mentioned earlier, the colorful polyhedron is isomorphic to the dual of the order complex (combinatorial barycentric subdivision) of the original polyhedron and can be realized as a map on the same underlying surface.

Flag adjacency graphs can also be defined for more general structures than polytopes, such as the incidence complexes studied in~\cite{ds,esch} (or even more general incidence geometries, as described in \cite{vin83a}). For an incidence complex of rank $n$ whose flags have $k_{i}-1$ (say) adjacent flags at level $i$ for each~$i$, its flag graph is again an $r$-regular graph, with $r:=k_{0}+\ldots+k_{n-1}$, whose edges can be labeled with the rank $i$ of the face in which the two corresponding flags differ. Then each color $i$ occurs exactly $k_{i}-1$ times at each vertex, so the edge labeling is not an edge coloring in the usual sense. Nevertheless, the flag graph still gives rise to an incidence structure of rank $n$, in much the same way as the flag graphs for abstract $n$-polytopes.

\section{Colorful polytopes from Cayley graphs}
\label{colcay}

In this section we introduce interesting families of colorful polytopes arising from certain types of Cayley graphs. We also revisit the graphicahedron (see \cite{graph}).

Let $\Gamma$ be a group with a distinguished set $\mathcal{T}:=\{\tau_1,\dots,\tau_n\}$ of (mutually distinct) involutory generators. Then the {\em Cayley graph of $\Gamma$ with respect to $\mathcal{T}$\/},
\[ \mathcal{G}=\mathcal{G}(\Gamma,\mathcal{T}),\]
is the connected $n$-regular graph with vertex set $\Gamma$ in which any two vertices $u$ and $v$ are adjacent if and only if $v=\tau_i u$ for some $i=1,\ldots,n$. There is a natural way of assigning colors from $R:=\{1,\ldots,n\}$ to the edges of $\mathcal{G}$, namely an edge $\{u,v\}$ receives color $i$ if $v=\tau_i u$. Then the $n$ edges emanating from a vertex of $\mathcal{G}$ all have different colors. Thus $\mathcal{G}$ is a properly (edge) colored $n$-regular graph with color set $R$. 

Theorem~\ref{mainthm} tells us that any such Cayley graph $\mathcal{G}$ gives rise to an abstract $n$-polytope $\mathcal{P}_{\mathcal{G}}$ whose $1$-skeleton is isomorphic to $\mathcal{G}$ and whose automorphism group is given by 
the group of color respecting automorphisms $\Gamma_{c}(\mathcal{G})$ of $\mathcal{G}$. This leads us to investigate 
$\Gamma_{c}(\mathcal{G})$.

First note that $\Gamma$ acts faithfully on $\mathcal{G}$ by right multiplication (with the inverse). More precisely, each element $g$ in $\Gamma$ induces a color preserving automorphism $\widehat{g}$ of $\mathcal{G}$ given (on the vertex set) by $\widehat{g}(u):=ug^{-1}$ for $u\in\Gamma$. Thus, if $\widehat{\Gamma}$ denotes the group of all such graph  automorphisms $\widehat{g}$ of $\mathcal{G}$ with $g\in\Gamma$, then $\widehat{\Gamma}$ is a subgroup of $\Gamma_p({\cal G})$ and hence of $\Gamma_c({\cal G})$. Clearly, $\widehat{\Gamma}$ is isomorphic to $\Gamma$.

Let $Aut(\Gamma, {\cal T})$ denote the group of all group automorphisms of $\Gamma$ that map the generating set $\mathcal{T}$ to itself and hence permute the distinguished generators of~$\Gamma$. By definition, each such group automorphism $d$ is a permutation of the vertex set $\Gamma$ of $\mathcal{G}$, and $d(\tau_{i}u) = d(\tau_i)d(u)$ for each $i=1,\ldots,n$ and $u\in\Gamma$. In particular, each edge $\{u,v\}$ of $\mathcal{G}$ of color $i$ is sent by $d$ to the edge $\{d(u),d(v)\}$ of $\mathcal{G}$ of color $\bar{d}(i)$, where $\bar{d}$ is the permutation of the subscripts $\{1,\ldots,n\}$ determined by $d(\tau_i)=\tau_{\bar{d}(i)}$ for each $i$. This shows that $d$ is a color respecting automorphism of~$\cal G$ with associated color permutation $\bar{d}$. Hence, $Aut(\Gamma, {\cal T})$ is also a subgroup of~$\Gamma_c({\cal G})$.

In fact, as the following theorem shows, the full color respecting automorphism group of $\cal G$ is a semidirect product of the  two special subgroups just described. For related results see also Jajcay~\cite{cayley}.

\begin{theorem}
\label{autocayley}
Let $\Gamma$ be a group with a distinguished set $\mathcal{T}:=\{\tau_1,\dots,\tau_n\}$ of involutory generators, and let $\mathcal{G}=\mathcal{G}(\Gamma,\mathcal{T})$ be its Cayley graph. Then
$\Gamma_c({\cal G}) = \Gamma \ltimes Aut(\Gamma, {\cal T}). $
\end{theorem}

\begin{proof}
We show that $\Gamma_c({\cal G})$ is an internal semi-direct product of its two subgroups $\widehat{\Gamma}$ and
$Aut(\Gamma, {\cal T})$. 

First we establish that together these subgroups generate $\Gamma_c({\cal G})$. Let $\gamma \in \Gamma_c({\cal G})$. Define $g$ in $\Gamma$ by $g:=\gamma(\varepsilon)$, where $\varepsilon$ is the unit element of $\Gamma$ (here viewed as a vertex of $\mathcal{G}$). Now consider the element $d:=\widehat{g}\,\gamma$ of $\Gamma_c({\cal G})$. We prove that its restriction to the vertex-set of $\mathcal{G}$, again denoted by $d$, is a group automorphism of $\Gamma$ permuting the generators in $\mathcal{T}$; that is, $d\in Aut(\Gamma, {\cal T})$. First note that 
\[ d(\varepsilon)=\widehat{g}(\gamma(\varepsilon))= \widehat{g}(g) = gg^{-1}=\varepsilon, \]
so $d$ fixes $\varepsilon$. Moreover, since $d$ lies in $\Gamma_c({\cal G})$, there exists a permutation $\bar{d}$ of the subscripts $\{1,\ldots,n\}$ such that $d(\tau_i)=\tau_{\bar{d}(i)}$ for each $i$; in fact, since $\varepsilon$ is fixed under $d$, each edge $\{\varepsilon,\tau_i\}$ of $\mathcal{G}$ with color $i$ is sent by $d$ to an edge $\{\varepsilon,\tau_{i'}\}$ of $\mathcal{G}$ with color $i'=:\bar{d}(i)$. In particular, $d$ permutes the generators in $\mathcal{T}$. More generally, for each $i$ and each $u\in\Gamma$, the edge $\{u,\tau_{i}u\}$ with color $i$ is mapped by the color respecting automorphism $d$ to the edge $\{d(u),d(\tau_{i}u)\} = \{d(u),\tau_{\bar{d}(i)}d(u)\}$ with color $\bar{d}(i)$, so in particular $d(\tau_{i}u)=d(\tau_{i})d(u)$. Since $\Gamma$ is generated by $\mathcal{T}$, this then implies that $d(uv)=d(u)d(v)$ for every $u,v \in \Gamma$, so $d$ is a homomorphism on $\Gamma$.  Hence $d \in Aut(\Gamma, {\cal T})$ and $\gamma=\widehat{g}^{-1}d\in \widehat{\Gamma}\cdot Aut(\Gamma, {\cal T})$. In particular, every element of $\Gamma_c({\cal G})$ can be written as a product of an element of $\widehat{\Gamma}$ and an element of $Aut(\Gamma, {\cal T})$.

Next we show that $\widehat{\Gamma}$ is normal in $\Gamma_c({\cal G})$. Since $\widehat{\Gamma}$ and $Aut(\Gamma, {\cal T})$
together generate $\Gamma_c({\cal G})$, it suffices to show that $\widehat{\Gamma}$ is normalized by $Aut(\Gamma, {\cal T})$.
Now, if $g,u\in\Gamma$ and $d \in Aut(\Gamma, {\cal T})$, then 
\[\begin{array}{llllllllll}
d\widehat{g}d^{-1}(u) &\!\!=\!\!& d(\widehat{g}(d^{-1}(u))) &\!\!=\!\!&d(d^{-1}(u)g^{-1}) \\
&&&\!\!=\!\!& d(d^{-1}(u))\, d(g^{-1}) &\!\! = \!\!& u\,d(g)^{-1} &\!\!=\!\!& \widehat{d(g)}(u). 
\end{array} \]
Hence $d\widehat{g}d^{-1}=\widehat{d(g)}$ for each $g\in\Gamma$. Thus $Aut(\Gamma,{\cal T})$ normalizes $\widehat{\Gamma}$. 

It remains to prove that the subgroups $\widehat{\Gamma}$ and $Aut(\Gamma, {\cal T})$ intersect trivially. Let $d\in \widehat{\Gamma}\cap Aut(\Gamma, {\cal T})$, and let $d=\widehat{g}$ with $g\in\Gamma$. Then
$\varepsilon = d(\varepsilon) = \widehat{g}(\varepsilon) = \varepsilon g^{-1}$,
so $g=\varepsilon$ and $d$ is the identity automorphism of $\mathcal{G}$. This completes the proof.
\end{proof}

A simple example of a colorful polytope associated with a Cayley graph is the $n$-dimensional cube. It is derived from the Cayley graph for the elementary abelian group $\mathbb{Z}_2^n$ with its $n$ canonical generators.

More exciting examples arise from Cayley graphs of symmetric groups. Recall from \cite{graph,symgra} that the graphicahedron associated with a given finite graph is an abstract polytope generalizing the well-known permutahedron. The permutahedron $\Pi_{n}$ can be described as the $n$-dimensional simple convex polytope whose $(n+1)!$ vertices are the points in $\mathbb{R}^{n+1}$ obtained from $(1,2,\ldots,n+1)$ by permutation of the coordinates. These vertices can be identified with the elements of the symmetric group $S_{n+1}$ in such a way that two vertices of $\Pi_{n}$ are connected by an edge if and only if the corresponding permutations differ by an adjacent transposition. The permutahedron was apparently first investigated by Schoute in 1911 (see \cite{S11,Z95}); it was rediscovered in Guilbaud \& Rosenstiehl~\cite{guro} in 1963 and given the name ``permutohedron'' (or ``permuto\`{e}dre'', in French). 
 
The construction of the $G$-graphicahedron associated with a given finite graph $G$ is based on the following Cayley graph derived from $G$. This Cayley graph (but not its respective graphicahedron) was also studied in Doignon \& Huybrechts~\cite{doignon}. 

Suppose $G$ is a finite simple graph with vertex set $V(G):=\{1,\ldots,p\}$ and edge set $E(G)=\{e_1,\ldots,e_q\}$, where $p\geq 1$ and $q\geq 0$ (if $q=0$ then $E(G)=\emptyset$). We associate with $G$ a Cayley graph on the symmetric group $S_p\,(=\Gamma)$ as follows. If $e = \{i,j\}$ is an edge of $G$, define ${\tau}_{e}:= (i\;j)$; this is the transposition in $S_p$ that interchanges $i$ and $j$. Let ${\cal T}_G:=\{\tau_{e_1},\ldots,\tau_{e_q}\}$ denote the set of transpositions determined by the edges of $G$, and let 
\[ {\cal G}_G:={\cal G}(S_p,{\cal T}_G)\] 
be the Cayley graph of $S_p$ with respect to ${\cal T}_G$. Then the vertex set of ${\cal G}_G$ is $S_p$, and $\{\gamma_1,\gamma_2\}$ is an edge of ${\cal G}_G$ if and only if $\tau_{e}\gamma_{1} =\gamma_2$ for some $e\in E(G)$. 

The Cayley graph ${\cal G}_{G}$ associated with a connected graph $G$ as above is a properly edge colored $q$-regular graph with color set $R=E(G)$. The corresponding colorful polytope $\mathcal{P}_{{\cal G}_G}$ is an abstract polytope of rank $q$ known as the {\em $G$-graphicahedron\/}, and $\mathcal{G}_G$ is its $1$-skeleton. 

If $G$ is a simple path with $q$ edges, then $\mathcal{P}_{{\cal G}_G}$ is just the $q$-dimensional permutahedron $\Pi_q$; this is a hexagon when $q=2$. More interestingly, if $G$ is a $q$-cycle, then $\mathcal{P}_{{\cal G}_G}$ is a tessellation of the $(q-1)$-dimensional torus by $(q-1)$-dimensional permutahedra intimately related to the geometry of the infinite euclidean Coxeter group $\widetilde{A}_{q-1}$ and the corresponding root lattice (see \cite{symgra}).

Recall from \cite[Theorem 5.1]{graph} that the polytope automorphism group of a general $G$-graphicahedron $\mathcal{P}_{{\cal G}_G}$ is given by $\Gamma({\cal P}_{{\cal G}_G}) = S_{p}\ltimes\Gamma(G)$. This is consistent with Theorems~\ref{automor} and \ref{autocayley}. In fact, we have
\[ \Gamma({\cal P}_{{\cal G}_G}) = \Gamma_{c}(\mathcal{G}_G) 
= S_p \ltimes Aut(S_p, {\cal T}_G) = S_{p}\ltimes\Gamma(G) .\]

It is worth noting that, if (and only if) $G$ is the complete graph on $p$ vertices, the mapping $\alpha\rightarrow \alpha^{-1}$ on $S_p$ induces a graph automorphism $\iota$ on the Cayley graph ${\cal G}_{G}$. This follows from the invariance of the edge set $E(G)$ under $\Gamma(G)=S_p$; in fact, if $\alpha,\beta\in S_p$, $e\in E(G)$, and $\beta=\tau_e\alpha$, then $\beta^{-1} = (\alpha^{-1}\tau_{e}\alpha)\alpha^{-1}=\tau_{\alpha^{-1}(e)}\alpha^{-1}$. However, this graph automorphism is not color respecting and hence does not give rise to a polytope automorphism; in fact, since the generators $\tau_e$ (the neighbors of $\varepsilon$ in $\mathcal{G}_G$) are invariant under $\iota$, the respective permutation of the color set $E(G)$ would necessarily have to be the identity permutation, which on the other hand leads to contradictions for the edge colors at other vertices. It was shown in \cite{doignon} that the full graph automorphism group of $\mathcal{G}_G$ in this case is $(S_{p}\ltimes S_p)\ltimes C_2$, with the factor $C_2$ generated by $\iota$.
\smallskip 

The degree of symmetry of a colorful polytope $\po_\mathcal{G}$ associated with a general Cayley graph $\mathcal{G}=\mathcal{G}(\Gamma,\mathcal{T})$ is determined by $Aut(\Gamma, {\cal T})$, the group of all group automorphisms of $\Gamma$ that permute the distinguished generators in~$\mathcal{T}$. In fact, by Theorems~\ref{automor} and~\ref{autocayley}, 
\begin{equation}
\Gamma(\mathcal{P}_{\mathcal{G}}) = \Gamma_{c}(\mathcal{G})
= \Gamma \ltimes Aut(\Gamma, {\cal T}) .
\end{equation}
This leads to the following interesting special case.

\begin{remark}
\label{rema}
When $Aut(\Gamma, {\cal T})$ is trivial, the group $\Gamma(\mathcal{P}_{\mathcal{G}})$ is isomorphic to $\Gamma$ and acts regularly on the vertices of $\mathcal{P}_{\mathcal{G}}$. The $G$-graphicahedra derived from a symmetry-free graph $G$ (with trivial graph automorphism group) are examples of colorful polytopes exhibiting this property.
\end{remark}

Other polytopes with a simply vertex-transitive action (that is, regular action on the vertices) are given by the polytopes $2^\mathcal{K}$, where $\mathcal{K}$ has trivial automorphism group (see \cite[8D]{McMS02} and \cite{psw}).

We remark that the symmetric groups that occur as automorphism groups of regular polytopes of various ranks were recently studied in Fernandes \& Leemans~\cite{ferlee}.

\section{The monodromy polytope}
\label{mono}

In this section we describe the colorful polytopes obtained from the Cayley graphs of the monodromy groups of polytopes.

Let $\mathcal{P}$ be an abstract $n$-polytope. The {\em universal} string Coxeter group $W := [\infty,\ldots,\infinite]$ of rank $n$, with distinguished involutory generators $s_0,s_1,\ldots,s_{n-1}$ and defining relations 
\[ s_{i}^{2} = (s_{i}s_{j})^{2} = \varepsilon \;\; \textrm{ for } i,j =0,\dots,n-1, \textrm{ with } i<j-1,\]
acts transitively on the set of flags $\fl(\po)$ of $\po$ (on the right) in such a way that $\Psi \cdot s_i = \Psi^i$, the $i$-adjacent flag of $\Psi$, for each $i=0,\dots,n-1$ and each $\Psi\in\fl(\po)$. In particular, if $w=s_{i_1}s_{i_2}\ldots s_{i_k} \in W$ then 
\[ \Psi \cdot w = (\Psi \cdot s_{i_1}s_{i_2}\ldots s_{i_{k-1}})\cdot s_{i_{k}} 
= (\Psi^{i_1,i_2,\ldots,i_{k-1}})^{i_k} =: \Psi^{i_1,i_2,\ldots,i_{k-1},i_k} \]
(see Hartley~\cite{hartley} and Hubard, Orbanic \& Weiss~\cite{d-auto}). This action defines a natural homomorphism 
\[\mu:W\rightarrow S_{\fl(\po)}\] 
from $W$ into the symmetric group $S_{\fl(\po)}$ on $\fl(\po)$. The {\em monodromy group} of $\po$, denoted $\Mon(\po)$, is the quotient of $W$ by the kernel $K$ of $\mu$, the normal subgroup of $W$ consisting of those elements that act trivially on $\fl(\po)$ (that is, fix every flag of $\po$). Then $\Mon(\po)$ is isomorphic to the image of $\mu$, and can be identified with this subgroup of $S_{\fl(\po)}$ whenever convenient. Let 
\[ \pi: W \to \Mon(\po)=W/K \] 
denote the canonical epimorphism. The transitive right action of $W$ on $\fl(\po)$ also induces a transitive right action of $\Mon(\po)$ on $\fl(\po)$ such that $\Psi\cdot {\pi(w)}=\Psi\cdot w$ for each $w\in W$ and $\Psi\in\fl(\po)$, and in particular $\Psi\cdot {\pi(s_i)}=\Psi^{i}$ for each $i$. We slightly abuse notation and let $s_i$ also denote the $i$-th generator $\pi(s_i)$ of $\Mon(\po)$ and $w$ also the element $\pi(w)=wK$ of $\Mon(\po)$. Observe that by definition of $K$ the action of $\Mon(\po)$ on $\fl(\po)$ is faithful, since only the unit element of $\Mon(\po)$ fixes every flag. 

The monodromy group $\Mon(\po)$ of an abstract polytope $\po$ is known to be isomorphic to the automorphism group $\Gamma(\po)$ of $\po$ if and only if $\po$ is a regular polytope (see \cite{hartley}). 

The monodromy group $\Mon(\po)$ and its generating set ${\cal T}=\{s_0,s_1,\ldots,s_{n-1}\}$ naturally give rise to a 
Cayley graph 
\[ \mathcal{M}:=\mathcal{G}(\Mon(\po),\mathcal{T}) \] 
associated with $\po$. Then Theorem~\ref{mainthm} says that $\mathcal{M}$ is the 1-skeleton of the corresponding colorful polytope $\po_{\mathcal{M}}$ of rank $n$, which we have named the {\em monodromy polytope\/} of $\po$. Its automorphism group $\Gamma(\po_{\mathcal{M}})$ is isomorphic to the group of color respecting automorphisms $\Gamma_{c}(\mathcal{M})$ of $\mathcal{M}$, by Theorem~\ref{automor}. Taking into account that the properly (edge) colored graph $\mathcal{M}$ is also a Cayley graph, we now can apply Theorem~\ref{autocayley} to conclude that $\Gamma(\po_{\mathcal{M}})$ is a semidirect product of $\Mon(\po)$ and $Aut(\Mon(\po),{\cal T})$, the group of automorphisms of $\Mon(\po)$ that permute the generators in $\cal T$. 

For an $n$-polytope $\po$, the admissible permutations of the generators $s_0,\ldots,s_{n-1}$ for group automorphisms of its monodromy group $\Mon(\po)$ are quite restricted. First observe that, for each $i=0,\ldots,n-2$, the product $s_{i}s_{i+1}$ of two consecutive generators of $\Mon(\po)$ has period $2$ if and only if every $i$-face of $\po$ is incident with every $(i+1)$-face of $\po$, or equivalently, if and only if every section of $\po$ of rank $2$ determined by an $(i-1)$-face and an $(i+2)$-face is a digon $\{2\}$. Thus, except in very degenerate situations, all products of consecutive generators of $\Mon(\po)$ have period greater than $2$. We shall require this property from now on. (Then, if $p_{i+1}$ denotes the period of $s_{i}s_{i+1}$ in $\Mon(\po)$ for $i=0,\ldots,n-2$, the string diagram of the underlying Coxeter group $[p_{1},\ldots,p_{n-1}]$ is connected.) As an immediate consequence, the elements $s_0$ and $s_{n-1}$ are distinguished among the generators (as representing the ends of the string), and any group automorphism of $\Mon(\po)$ that permutes the generators must necessarily fix or interchange $s_0$ and $s_{n-1}$ and then be uniquely determined. Hence any non-trivial group automorphism of $\Mon(\po)$ that permutes the generators must necessarily reverse the order of the generators. It follows that $Aut(\Mon(\po),{\cal T})$ is trivial or is a group of order $2$.

In summary, we have established the following theorem. Recall the notion of an {\em $(i,i+1)$-face layer graph\/} from Section~\ref{flagpo}.

\begin{theorem}
Let $\po$ be an $n$-polytope such that none of its $(i,i+1)$-face layer graphs, with $i=0,\ldots,n-2$, is a complete bipartite graph. Let $\po_{\mathcal{M}}$ be the monodromy polytope, the colorful polytope arising from the Cayley graph $\mathcal{M}$ of the monodromy group $\Mon(\po)$ of $\po$ with its canonical generators $s_0, \dots s_{n-1}$. Then $\Gamma(\po_{\mathcal{M}})$ is isomorphic to  $\Mon(\po) \ltimes C_2$ or $\Mon(\po)$, according as $\Mon(\po)$ does, or does not, admit a group automorphism sending $s_i$ to $s_{n-i-1}$ for each $i=1,\ldots,n$.
\end{theorem}

We conclude this section by establishing a covering relationship between the monodromy polytope and the flag adjacency polytope of a given polytope $\po$. This is based on a similar relationship between the Cayley graph $\mathcal{M}$ of the monodromy group of $\po$, and the flag graph of $\po$ described in Section~\ref{flagpo}. We begin by studying the graph covering.

Recall that a graph with a vertex partition invariant under some given subgroup of its graph automorphism group naturally gives rise to a new graph, a quotient graph, whose vertices are the members of the partition and whose edges join two members of the partition precisely when they contain vertices joined in the original graph. Then the original graph is said to be an {\em imprimitive cover\/} of the new graph (see \cite{praeger}).

\begin{propo}
\label{covgr}
Let $\po$ be a polytope of rank $n$, let $\cal G$ be the flag graph of $\po$, and let $\cal M$ be the Cayley graph of the monodromy group $\Mon(\po)$ of $\po$ with canonical generators~$s_0,\dots s_{n-1}$. Then $\cal M$ is an imprimitive cover of $\cal G$ (as a properly (edge) colored graph). In particular, ${\cal M}$ and $\cal G$ are isomorphic (as properly (edge) colored graphs) if and only if $\po$ is a regular polytope.
\end{propo}

\begin{proof}
Recall that the vertices of $\mathcal{G}$ are just the flags of $\po$, with an edge of color~$i$ between two vertices of $\mathcal{G}$ if and only if the corresponding flags of $\mathcal{P}$ are $i$-adjacent. On the other hand, the vertices of $\mathcal{M}$ are just the elements of $\Mon(\po)$, with an edge of color $i$ between two vertices $u$ and $v$ of $\mathcal{M}$ if and only if $v= s_{i}u$ in $\Mon(\po)$; the latter just says that $\Psi\cdot v  = \Psi^{i}\cdot u$ for each flag $\Psi$ of $\po$. However, in general the action of $\Mon(\po)$ on $\fl(\po)$ need not be free (semi-regular), indicating that $\cal M$ might have more vertices than $\cal G$. In fact, $\Mon(\po)$ acts freely on $\fl(\po)$ (and then $\cal M$ and $\cal G$ are isomorphic) if and only if $\po$ is a regular polytope.

In general the relationship between $\mathcal{M}$ and $\mathcal{G}$ can be described as follows. Let $\Phi$ be a fixed, or {\em base}, flag of $\po$, and let $H$ denote the stabilizer of $\Phi$ under the action of $\Mon(\po)$ on $\fl(\po)$. Clearly, the left cosets of $H$ in $\Mon(\po)$ form an $H$-invariant partition of $\Mon(\po)$, the vertex set of $\cal M$; note here that $H$ acts on $\mathcal{M}$ (by multiplication with the inverse on the right) as a group of graph automorphisms. 

Now consider the corresponding quotient graph ${\cal M}_H$ of $\mathcal{M}$ determined by $H$. The vertices of ${\cal M}_H$ are the left cosets $uH$ with $u \in \Mon(\po)$, and $\{uH,vH\}$ is an edge of $\mathcal{M}_H$ if and only if there exist $u'$ and $v'$ in $\Mon(\po)$ such that $uH=u'H$, $vH=v'H$, and $\{u',v'\}$ is an edge of $\cal M$ (that is, $v'=s_{i}u'$ in $\Mon(\po)$ for some~$i$). The new graph $\mathcal{M}_H$ inherits a natural edge coloring from $\mathcal{M}$, with the same color set $\{0,\ldots,n-1\}$ as $\mathcal{M}$, making it a properly (edge) colored graph; in fact, if $\{uH,vH\}$ is an edge of $\mathcal{M}_H$ and $\{u',v'\}$ as above is an edge of $\mathcal{M}$ of color $i$, then we can safely assign the color $i$ to $\{uH,vH\}$. Note here that the color of an edge $\{uH,vH\}$ is indeed well-defined; that is, if $\{uH,vH\}$ is represented by another pair $u'',v''$ such that $uH=u''H$, $vH=v''H$, and $\{u'',v''\}$ is an edge of $\cal M$ of color $j$, then necessarily $j=i$. This is implied by the definition of $H$, as we will see shortly. Then it will be clear that $\mathcal{M}$ is an imprimitive cover of $\mathcal{M}_H$.

Next observe that the mapping 
\begin{equation}
\label{bij}
uH \rightarrow \Phi\cdot u^{-1} \quad (u\in \Mon(\po)) 
\end{equation}
is a well-defined bijection $\varphi$ between the vertex-sets of $\mathcal{M}_H$ and $\mathcal{G}$. Bear in mind here that $H$ is the stabilizer of the base flag $\Phi$ under the action of $\Mon(\po)$, so in particular $\Phi \cdot u^{-1}= \Phi\cdot u'^{-1}$ if $u,u'\in\Mon(\po)$ and $uH=u'H$. 

This bijection between the vertex sets naturally extends to a full graph automorphism $\phi: {\cal M}_H \rightarrow \mathcal{G}$ between $\mathcal{M}_H$ and $\mathcal{G}$, which again was denoted by $\varphi$. To see this, let $\{uH,vH\}$ be an edge of ${\cal M}_H$ of color $i$, and let $\{u',v'\}$ as above be a corresponding edge of $\mathcal{M}$ of color $i$ so that $v'=s_{i}u'$ in $\Mon(\po)$. Then  
\begin{equation}
\label{bijedge}
\begin{array}{lllllllllllll}
\varphi(vH)&\!\!\!=\!\!\!& \!\Phi\cdot v^{-1} &\!\!\!=\!\!\!&\! \Phi\cdot v'^{\,-1} 
&\!\!\!=\!\!\!&\Phi\cdot (u'^{\,-1}s_{i})&\!\!\!=\!\!\!&(\Phi \cdot u'^{\,-1})\cdot s_i\\[.05in]
&&&&&\!\!\! =\! \!\!&(\varphi(u'H))\cdot s_i  &\!\!\!=\!\!\!& (\varphi(u'H))^i &\!\!\!\!=\!\!\!&(\varphi(uH))^i.
\end{array} 
\end{equation}
Hence, if $\{uH,vH\}$ is an edge of ${\cal M}_H$ of color $i$, then $\varphi(vH)$ and $\varphi(uH)$ are $i$-adjacent flags of 
$\po$ and hence are joined by an edge of $\mathcal{G}$ of color $i$, the image of $\{uH,vH\}$ under $\varphi$. Moreover, by  the connectedness properties of $\po$, each edge of $\cal G$ of color $i$ is the image of an edge of ${\cal M}_H$ of color $i$ under~$\varphi$. Thus $\varphi$ is a graph isomorphism. 

Very similar arguments also complete the proof that the edge coloring of $\mathcal{M}_H$ inherited from $\mathcal{M}$ is indeed well-defined. In fact, if $u',u'',v',v''$ as above are such that $v'=s_{i}u'$ and $v''=s_{j}u''$ in $\Mon(\po)$ for some $i$ and $j$, then the corresponding equations in (\ref{bijedge}), first applied with $u',v'$ and then with $u'',v''$, show that $\varphi(vH)$ and $\varphi(uH)$ are a pair of flags of $\po$ that are both $i$-adjacent and $j$-adjacent, and hence that $j=i$.

Thus $\mathcal{M}_H$ and $\mathcal{G}$ are isomorphic, and $\mathcal{M}$ is an imprimitive cover of $\mathcal{G}$ associated with $H$.
\end{proof}

Recall that a surjective mapping $\gamma:\mathcal{Q}\rightarrow\mathcal{R}$ between two polytopes $\mathcal{Q}$ and $\mathcal{R}$ of the same rank is called a {\em covering\/} if $\gamma$ preserves incidence of faces in one direction (incidence in $\mathcal{Q}$ implies incidence of images in $\mathcal{R}$), ranks of faces, and adjacency of flags (see \cite[p. 43]{McMS02}). If there exists such a covering $\gamma:\mathcal{Q}\rightarrow\mathcal{R}$ for polytopes $\mathcal{Q}$ and $\mathcal{R}$, then we also say that $\mathcal{Q}$ is a {\em covering\/} of $\mathcal{R}$. 

\begin{theorem}
Let $\po$ be a polytope of rank $n$, let $\cal G$ be the flag graph of $\po$, and let $\cal M$ be the Cayley graph of the monodromy group $\Mon(\po)$ of $\po$ with canonical generators~$s_0,\dots s_{n-1}$. Then the monodromy polytope $\po_{\mathcal{M}}$ of $\po$ is a covering of the flag adjacency polytope $\po_{\mathcal{G}}$ of $\po$. In particular, $\po_{\mathcal{M}}$ and $\po_{\mathcal{G}}$ are isomorphic if and only if $\po$ is a regular polytope.
\end{theorem}

\begin{proof} 
We know from Proposition~\ref{covgr} and its proof that $\mathcal{G}$ is isomorphic to the quotient graph $\mathcal{M}_H$ of $\mathcal{M}$ defined by the subgroup $H$ of $\Mon(\po)$ described earlier. Hence the colorful polytopes $\po_{\mathcal{G}}$ of $\mathcal{G}$ and $\po_{\mathcal{M}_H}$ of $\mathcal{M}_H$ are isomorphic. 

Consider the graph covering $\gamma: \mathcal{M}\rightarrow \mathcal{M}_H$ defined (on the vertex-sets) by $\gamma(u):=uH$ for $u\in\Mon(\po)$. Then $\gamma$ is color preserving, meaning that an edge of $\mathcal{M}$ is mapped to an edge of $\mathcal{M}_H$ of the same color. This graph covering extends in a natural way to a covering, again denoted by $\gamma$, between the colorful polytopes of these two graphs. More specifically, $\gamma: \po_{\mathcal{M}}\rightarrow \po_{\mathcal{M}_H}$ is given by 
\[ \gamma((C,u)) := (C,uH) \qquad (C\subseteq R,\, u\in \Mon(\po)) ,\]
with $R:= \{0,\ldots,n-1\}$. Then $\gamma$ is well-defined and preserves incidence in one direction. In fact, if $C,D\subseteq R$, $u,v\in \Mon(\po)$, and $(C,u)\leq (D,v)$ in $\po_{\mathcal{M}}$, then $C\subseteq D$ and $u,v$ can be joined in $\mathcal{M}$ by a path of edges with colors from $D$; but then $uH,vH$ can be joined by a path in $\mathcal{M}_H$ with edge labels from $D$, so $(C,uH)\leq (D,vH)$ in $\po_{\mathcal{M}_H}$. In particular, when $C=D$ this proves that $\gamma$ is well-defined. Moreover, $\gamma$ is surjective, rank preserving, and flag adjacency preserving. For the latter observe that 
$\gamma$ takes a flag $\Psi = (\mathcal{C},u)$ of $\po_{\mathcal{M}}$, with $u\in \Mon(\po)$ and $\mathcal{C}:= \{C_{0},C_{1},\ldots,C_{n}\}$ a maximal nested family of subsets of $R$, to the flag $\gamma(\Psi) = (\mathcal{C},uH)$ of $\po_{\mathcal{M}_H}$. Thus $\gamma: \po_{\mathcal{M}}\rightarrow \po_{\mathcal{M}_H}$ is a covering of polytopes.

Finally, being colorful polytopes, $\po_{\mathcal{M}}$ and $\po_{\mathcal{M}_H}$ are isomorphic if and only if their $1$-skeletons are isomorphic edge-colored graphs. By Proposition~\ref{covgr} we know this to happen if and only if $\po$ is a regular polytope.
\end{proof}
\bigskip

\noindent
{\bf Acknowledgment:} We are very grateful to Javier Bracho and Luis Montejano for bringing their article~\cite{bracho} on colored triangulations of manifolds to our attention.

\end{document}